\documentclass[12pt,leqno]{amsart}
\usepackage[all]{xy}
\usepackage{amssymb,amsmath,amsthm,mathrsfs}
\usepackage{cases}
\usepackage{enumerate}
\usepackage{hyperref}
\usepackage{bm}

\numberwithin{equation}{section}
\allowdisplaybreaks[3]
\makeatother
\newtheorem{thm}{Theorem}[section]

\newtheorem{lem}[thm]{Lemma}

\theoremstyle{definition}

\newtheorem{rem}[thm]{Remark}
\newtheorem{exa}[thm]{Example}

\numberwithin{equation}{section}

\theoremstyle{definition}
\newtheorem*{ack}{Acknowledgements}

\newcommand{\eps}{\varepsilon}

\newcommand{\si}{\sigma}

\newcommand{\re}{\textup{Re}}
\newcommand{\im}{\textup{Im}}

\newcommand{\abs}[1]{\left\lvert#1\right\rvert}

\begin{document}
\title[On the moments of averages of quadratic twists of the M\"obius function]{On the moments of averages of quadratic twists of the M\"obius function} 
\author[Y. Toma]{Yuichiro Toma}
\address{Graduate School of Mathematics, Nagoya University, Chikusa-ku, Nagoya 464-8602, Japan.}
\email{m20034y@math.nagoya-u.ac.jp}

\makeatletter
\@namedef{subjclassname@2020}{\textup{2020} Mathematics Subject Classification}
\makeatother

\begin{abstract}
We consider the moment of quadratic twists of the M\"obius function of the form 
\[
 S_k(X,Y) = \sum_{d\leq X} \left( \sum_{n\leq Y} \left(\frac{8d}{n}\right) \mu(n)\right)^k,
\]
where $\left(\frac{8d}{\cdot}\right)$ is the Kronecker symbol and $d$ runs over positive, odd and square-free integers. We give unconditional results for their asymptotic behaviors. 
\end{abstract} 
\maketitle
\section{Introduction}
Let $\mu$ be the M\"obius function. The Mertens function is defined by
\[
    M(x) := \sum_{n \leq x} \mu(n).
\]
Classically, it is known that $M(x)$ is deeply connected with the Riemann hypothesis (RH). Indeed, the RH is equivalent to $M(x) \ll x^{1/2+\eps}$ for any $\eps>0$ (cf. \cite{So09}). Under the RH, the bound of $M(x)$ was considered by Landau~\cite{L24}, Titchmarsh~\cite{Tit27}, Maier and Montgomery~\cite{MM09} and Soundararajan~\cite{So09}. By refining Soundararajan's method~\cite{So09}, Balazard and de Roton~\cite{BdeR} showed 
\[
M(x) \ll x^\frac{1}{2}\exp\left( (\log x)^\frac{1}{2}(\log\log\log x)^{\frac{5}{2}+\eps}\right).
\]

As a generalization of the Mertens function, for a Dirichlet character $\chi$ modulo $q$, the corresponding generalized Mertens function is defined by
\begin{align}
    \label{twisted Mobius function}
    M(x,\chi) := \sum_{n \leq x} \mu(n) \chi(n).
\end{align}

A similar relationship between the RH and the Mertens function is known for (\ref{twisted Mobius function}). It is related to the generalized Riemann hypothesis (GRH) of the corresponding Dirichlet $L$-function $L(s, \chi)$. Under the GRH, the size of (\ref{twisted Mobius function}) was studied by Halupczok and Suger~\cite{HS} and their result is further refined by Ye~\cite{Ye} which asserts that under the GRH, 
\[
M(x,\chi) \ll x^\frac{1}{2} \exp\left( (\log x)^{\frac{3}{5}+o(1)} \right)
\]
uniformly for $q$ and $x$.

Moreover, unlike the Mertens function, the size of (\ref{twisted Mobius function}) with primitive quadratic characters is affected by the existence or nonexistence of the Siegel zero. 

We write $\chi_d$ for the Kronecker symbol $\left(\frac{d}{\cdot}\right)$. If $d$ is odd and square-free, then $\chi_{8d}$ is a primitive Dirichlet character. Then we consider moments of (\ref{twisted Mobius function}). Let 
\begin{align}
\label{S_k}
    S_k(X,Y) &:= \sideset{}{^*}{\sum}_{0<d\leq X} \left( \sum_{n\leq Y} \chi_{8d}(n) \mu(n)\right)^k
\end{align}
where the asterisk indicates that $d$ runs over odd and square-free integers.

In the case $k=2$, (\ref{S_k}) was introduced by Gao and Zhao~\cite{GZ23}. They considered smoothed sums instead, not (\ref{S_k}) itself. For fixed smooth functions $\Phi(x), W(x)$ that are compactly supported on $(0,\infty)$, they set 
\[
S(X,Y;\Phi,W):=\sideset{}{^*}{\sum}_{d} \left( \sum_{n} \chi_{8d}(n) \mu(n) \Phi\left(\frac{n}{Y}\right)\right)^2 W\left(\frac{d}{X}\right).
\]
Then, they evaluated $S(X,Y;\Phi,W)$ asymptotically.

\begin{thm}[Theorem 1.1 in \cite{GZ23}]
\label{thm:GZ}
With assuming the truth of the GRH, for large $X$ and $Y$ with $Y \ll X^{1-\eps}$, we have
\begin{align*}
S(X,Y;\Phi,W) &= \frac{4}{\pi^2} \tilde{h}(1,1)Z_2XY +O(X^{\frac{1}{2}+\eps}Y^{\frac{3}{2}+\eps}) +O(XY^{\frac{1}{2}+\eps}),
\end{align*}
where 
\begin{align}
\label{ccontribution of Mellin transform}
    \tilde{h}(1,1) &= \int_\mathbb{R} W(x)dx\left( \int_\mathbb{R} \Phi(y) dy\right)^2,
\end{align}
and 
\begin{align}
    \label{Z_2(1)}
    Z_2&= \prod_{p}\left(1-\frac{2}{p(p+1)}\right).
\end{align}
\end{thm}

In this article we give an asymptotic formula of $S_2(X,Y)$ when the sizes of $X,Y$ satisfy $Y \ll X^{1/3-\eps}$ without assuming the GRH . 

\begin{thm}
\label{thm:1}
For large $X$ and $Y$ with $Y \ll X^{1/3-\eps}$, we have
\begin{align*}
    S_2(X,Y) &= \frac{4}{\pi^2}Z_2 XY+ O(XY^{\frac{7}{12}})+O(X^{\frac{1}{2}+\eps}Y^{2+\eps}) +O(X^\frac{1}{2}Y^{\frac{5}{2}}(\log Y)^2).
\end{align*}
\end{thm}
In this restriction $Y \ll X^{1/3-\eps}$, the main term in Theorem \ref{thm:1} coincides with the main term in Theorem \ref{thm:GZ} except for (\ref{ccontribution of Mellin transform}). The third error term $O(X^\frac{1}{2}Y^{\frac{5}{2}}(\log Y)^2)$ comes from the P\'olya-Vinogradov inequality (see Lemma \ref{lem:RudSou} and (\ref{S_k with f})). 

Next, to state our result for $k \geq 3$, we introduce $q = \binom{k}{2}$ linear forms $\left( l^{(1)}, \dots, l^{(q)}\right)$ from $\mathbb{C}^k$ to $\mathbb{C}$ whose restriction to $\mathbb{R}_{\geq 0}$ have values in $\mathbb{R}$ as follows. Let $\bm{x}=(x_1,\dots,x_k) \in \mathbb{C}^k$ and $I_{n}$ be the $n \times n$ identity matrix. We define 
\begin{align*}
    \left( l^{(1)}(\bm{x}), \dots, l^{(q)}(\bm{x})\right) &= (x_1,\dots,x_k) 
    \begin{pmatrix}
    \begin{matrix} 1 & \cdots & 1 \\ & & \\  & \text{\large{$I_{k-1}$}} & \\ & & \\ & & \end{matrix} & \begin{matrix} 0 &\cdots & 0 \\ 1 & \cdots & 1 \\ & & \\ & \text{\large{$I_{k-2}$}} & \\ & & \end{matrix} & \begin{matrix} 0 & \dots & 0 \\  0 & \cdots & 0 \\ 1 & \cdots & 1 \\ & & \\ & \text{\large{$I_{k-3}$}} & \end{matrix} & \begin{matrix} \cdots \\  \cdots \\  \\ & \\ \cdots &  \end{matrix} 
\end{pmatrix}.
\end{align*}
i.e., the above $k \times q$ matrix consists of all column vectors of length $k$ with two entries equal to $1$ and all other entries equal to $0$.

Then, we obtain the following result.

\begin{thm}
\label{thm:2}
For large $X$ and $Y$ with $Y \ll X^{2/3k-\eps}$, we have
\begin{align*}
S_k(X,Y) &= \frac{4}{\pi^2}Z_kXY^\frac{k}{2}Q(\log Y)+ O(XY^{\frac{k}{2}-\theta})+O(X^{\frac{1}{2}+\eps}Y^{k+\eps}) +O(X^\frac{1}{2}Y^{\frac{5k}{4}}(\log Y)^k),
\end{align*}
where $Q \in \mathbb{R}[x]$ is a monic polynomial of exact degree $\binom{k}{2}-k$, $\theta>0$ and $Z_k$ is given by
\begin{align*}
Z_k &=\prod_{p} \left( 1-\binom{k}{2} \frac{1}{p(p+1)} + \sum_{h=1}^{\lfloor \frac{k}{2} \rfloor} \sum_{\ell = 1}^{\binom{k}{2}} (-1)^{\ell}\binom{k}{2h} \frac{1}{p^{h+\ell-1}(p+1)} \right) \mathcal{I},
\end{align*}
where 
\begin{align*}
    \mathcal{I} &= \lim_{x \to \infty} \frac{1}{x^{\frac{k}{2}}(\log x)^{q-k}}\int_{\mathcal{A}(x)} dy_1\dots dy_q,
\end{align*}
with  and
\[
\mathcal{A}(x) = \left\{ (y_1,\dots,y_q) \in [1,\infty)^q \middle| \prod_{i=1}^q y_i^{l^{(i)}(\bm{e}_j)} \leq x \text{ for all } 1 \leq j \leq k \right\}.
\]
\end{thm}

\begin{exa}
In the case $k=3$, we find that 
\begin{align*}
\mathcal{I} = \lim_{x \to \infty} \frac{1}{x^{\frac{3}{2}}}\int_{\substack{y_1,y_2,y_3 \geq 1 \\ y_1 y_2 \leq x,  y_1 y_3 \leq x, y_2 y_3 \leq x}} dy_1 dy_2 dy_3 = \lim_{x \to \infty} \frac{1}{x^{\frac{3}{2}}}(4x^{\frac{3} {2}}-6x+3x^{\frac{1}{2}}-1)=4.
\end{align*}
Therefore, we have
\[
Z_3 = 4\prod_{p} \left( 1-\frac{6}{p(p+1)}+\frac{3}{p^2(p+1)}-\frac{3}{p^3(p+1)} \right).
\]
\end{exa}

In \cite{GZ23}, they applied the techniques given by Soundararajan~\cite{So00} in his work on the nonvanishing of the central values of quadratic Dirichlet $L$-functions. They utilized Soundararajan's method and evaluated the contribution coming from the weight functions $\Phi,W$ by applying the result developed by Soundararajan and Young \cite{SY}. Assuming the truth of GRH, they can extend the region where the asymptotic formula for $S(X,Y;\Phi,W)$ holds.

On the other hand, we treat a weight removed character sum and give the asymptotic behavior of $S_k(X,Y)$ in restricted region $Y \ll X^{2/3k-\eps}$. To do this, we apply the estimates of quadratic character which was developed by Rudnick and Soundararajan~\cite{RS} (see Lemma \ref{lem:RudSou} below). It enables us to transform averages of quadratic twists of the M\"obius function into an arithmetic function of several variables. 

The study of arithmetic functions of several variables is applicable to the moments problem of single variable arithmetic functions. By applying analytical methods of multiple Dirichlet series, the moments problems are studied such as the M\"obius sums~\cite{BNP} and the Ramanujan sums \cite{CK}, \cite{GRam M}. 

In \cite{GRam M}, Goel and Ram Murty gave the asymptotic formula for the average of the moments of the Ramanujan sums by using the de la Bret\`{e}che multivariable Tauberian Theorem. Similar to their case, we can apply the result of de la Bret\`{e}che to estimate (\ref{S_k}). This will be shown in Section \ref{section5}, and properties of arithmetical functions of several variables will be described in the next section.

\section{Arithmetical functions of several variables}
For a natural number $k$, let $f : \mathbb{N}^k \to \mathbb{C}$ be an arithmetic function of $k$ variables. Then $f$ is said to be multiplicative if it is not identically zero and
\[
f(m_1n_1,\dots,m_kn_k) = f(m_1,\dots,m_k)f(n_1,\dots,n_k)
\]
holds for any $m_1,\dots,m_k, n_1,\dots,n_k \in \mathbb{N}$ such that $\gcd(m_1 \dots m_k, n_1 \dots n_k) = 1$. This is a generalization of the classical one-variable multiplicative functions which are not identically zero and $f(mn) = f(m)f(n)$ for $\gcd(m,n)=1$. 

If $f$ is multiplicative, then it is determined by the values $f(p^{v_1}, \dots, p^{v_k})$, where $p$ is prime and $v_1, \dots,v_k \in \mathbb{N}_0$. Thus, a formal Dirichlet series of several variables of $f$ admits an Euler product expansion:
\begin{align}
\label{Euler prod}
\sum_{n_1, \dots n_k =1}^\infty \frac{f(n_1,\dots n_k)}{n_1^{s_1}\dots n_k^{s_k}} =\prod_{p} \left( \sum_{v_1,\dots,v_k = 0}^\infty \frac{f(p^{v_1}, \dots, p^{v_k})}{p^{v_1s_1+\dots+v_ks_k}}\right).
\end{align}

In our context, we consider the function
\begin{align}
\label{f k-variable}
    f(n_1,\dots,n_k) = \mu(n_1)\dots\mu(n_k) \prod_{p|n_1\dots n_k} \frac{p}{p+1} \times \bm{1}_{n_1\dots n_k = \square}.
\end{align}
Here, we write $n=\square$ for $n$ being a square number. For $m_1,\dots,m_k, n_1,\dots,n_k \in \mathbb{N}$ such that $\gcd(m_1 \dots m_k, n_1 \dots n_k) = 1$, we see that
\[
\bm{1}_{m_1n_1\dots m_kn_k = \square} = \bm{1}_{m_1\dots m_k = \square} \bm{1}_{n_1\dots n_k = \square}
\]
since $m_1n_1\dots m_kn_k = \square$ if and only if $m_1\dots m_k = \square$ and $n_1\dots n_k = \square$. It is easy to check that the function $\prod_{p|n_1\dots n_k} \frac{p}{p+1}$ is also multiplicative. Therefore we find that the function $f(n_1,\dots,n_k)$ is multiplicative.

\begin{lem}
\label{lem:convolution}
    We have   
    \begin{align}
        \sum_{n_1, \dots, n_k =1}^\infty \frac{f(n_1,\dots n_k)}{n_1^{s_1}\dots n_k^{s_k}} &= \left(\prod_{\substack{I \subseteq \{1,\dots,k\} \\ \abs{I}=2 }}\zeta(s_I)\right) E(s_1,\dots,s_k), 
    \end{align}
    where $s_I:= s_{\ell_1}+\dots +s_{\ell_r}$ for any subset $I = \{ \ell_1,\dots,\ell_r \}$ of $\{1,\dots,k\}$, and $E(s_1,\dots,s_k)$ is a Dirichlet series absolutely convergent for $\re(s_j) > 1/4$ for $1\leq j \leq k$.
\end{lem}
\begin{proof}
Form (\ref{Euler prod}) and (\ref{f k-variable}), we have 
\begin{align}
\label{euler-prod}
    \sum_{n_1, \dots n_k =1}^\infty \frac{f(n_1,\dots n_k)}{n_1^{s_1}\dots n_k^{s_k}} 
    &= \prod_{p} \left( 1+\sum_{\substack{v_1,\dots,v_k \in \{ 0,1 \} \\ (v_1,\dots,v_k)\neq (0,\dots,0) \\ \abs{\bm{v}}:\text{even} \\ \abs{\bm{v}} \neq 0}} \left(\frac{p}{p+1}\right) \frac{1}{p^{v_1s_1+\dots+v_ks_k}} \right),
\end{align}
where $\abs{\bm{v}} = \# \{ j \in \{1,\dots,k\} \mid v_j=1 \text{ for } v_1,\dots,v_k \in \{ 0,1 \} \}$. Therefore we can factor the Euler product to
get
\begin{align*}
    & \prod_{p} \left( 1+\sum_{\substack{v_1,\dots,v_k \in \{ 0,1 \} \\ (v_1,\dots,v_k)\neq (0,\dots,0) \\ \abs{\bm{v}}:\text{even} \\ \abs{\bm{v}}\neq 0}} \left(\frac{p}{p+1}\right) \frac{1}{p^{v_1s_1+\dots+v_ks_k}} \right) \\
    &= \left( \prod_{\substack{I \subseteq \{1,\dots,k\} \\\abs{I}=2}} \zeta(s_I) \right) E(s_1,\dots,s_k),
\end{align*}
say. The term $E(s_1,\dots,s_k)$ can be calculated as
\begin{align}
    \label{E(s)}
    E(s_1,\dots,s_k) &= \left(\prod_{p} \left( 1+\sum_{\substack{I \subseteq \{1,\dots,k\} \\ \abs{I}:\text{even} \\ \abs{I} \neq 0}} \left(\frac{p}{p+1}\right) \frac{1}{p^{s_I}} \right)\right) \prod_{\substack{I \subseteq \{1,\dots,k\} \\\abs{I}=2}}\prod_{p} \left( 1-\frac{1}{p^{s_I}}\right) \\
    &=\prod_{p} \left( 1+\sum_{\substack{I \subseteq \{1,\dots,k\} \\\abs{I}=2}} \left( \left(\frac{p}{p+1}\right) -1 \right)\frac{1}{p^{s_I}} \right. \nonumber \\ 
    &\qquad\qquad \left.+ \sum_{\ell = 1}^{\binom{k}{2}} (-1)^{\ell}\sum_{\substack{I, J_1,\dots,J_{\ell} \subseteq \{1,\dots,k\} \\ \abs{I}:\text{even}, \ \abs{I} \neq 0 \\ \abs{J_1},\dots,\abs{J_{\ell}}=2}} \left(\frac{p}{p+1}\right) \frac{1}{p^{s_I+s_{J_1}+\dots+s_{J_\ell}}} \right) \nonumber \\
    &=\prod_{p} \left( 1-\sum_{\substack{I \subseteq \{1,\dots,k\} \\\abs{I}=2}} \frac{1}{p+1}\frac{1}{p^{s_I}} 
    \right. \nonumber \\
    &\qquad\qquad \left.+ \sum_{\ell = 1}^{\binom{k}{2}} (-1)^{\ell}\sum_{\substack{I, J_1,\dots,J_{\ell} \subseteq \{1,\dots,k\} \\ \abs{I}:\text{even}, \ \abs{I} \neq 0 \\ \abs{J_1},\dots,\abs{J_{\ell}}=2}} \left(\frac{p}{p+1}\right) \frac{1}{p^{s_I+s_{J_1}+\dots+s_{J_\ell}}} \right) \nonumber.
\end{align}
For any even $s_I,s_J$ such that $\abs{I},\abs{J}\neq 0$, 
\[
\re(s_I)+\re(s_J)>1
\]
holds for $\re(s_j)>1/4$. Therefore, the above Euler product is convergent absolutely for $\re(s_j)>1/4$. So we obtain the desired result.
\end{proof} 

\section{Estimation}
We treat the relevant character sum individually so that we apply the following estimates of quadratic characters given by Rudnick and Soundararajan~\cite{RS}, which are the major ingredients of this article.

\begin{lem}
\label{lem:RudSou}
    Let $n$ be an odd integer, and let $z \geq 3$ be a real number. If $n$ is not a perfect square then
    \[
    \sum_{d \leq z}\mu(2d)^2\left(\frac{8d}{n}\right) \ll z^{\frac{1}{2}}n^{\frac{1}{4}}\log (2n),
    \]
    while if $n$ is a perfect square then
    \[
    \sum_{d \leq z}\mu(2d)^2\left(\frac{8d}{n}\right) =\frac{z}{\zeta(2)} \prod_{p|2n} \frac{p}{p+1}+O(z^{\frac{1}{2}+\eps}n^\eps).
    \]
\end{lem}
By using the above estimates, we obtain
\begin{align*}
    S_k(X,Y) &= \sum_{\substack{0<d<X \\ (2,d)=1}} \mu(2d)^2 \left( \sum_{n\leq Y} \chi_{8d}(n) \mu(n)\right)^k \\
    &= \sum_{\substack{0<d<X \\ (2,d)=1}} \mu(2d)^2 \sum_{n_1\leq Y} \dots \sum_{n_k\leq Y} \chi_{8d}(n_1) \mu(n_1) \dots \chi_{8d}(n_k) \mu(n_k) \\
    &= \sum_{n_1\leq Y} \dots \sum_{n_k\leq Y} \mu(n_1) \dots \mu(n_k)\sum_{\substack{0<d<X \\ (2,d)=1}} \mu(2d)^2 \chi_{8d}(n_1 \dots n_k).
\end{align*}
Then we apply Lemma \ref{lem:RudSou} and factor out $p=2$ to find that 
\begin{align}
\begin{split}
\label{S_k with f}
    S_k(X,Y) &=\frac{X}{\zeta(2)} \sum_{\substack{n_1\leq Y, \dots, n_k\leq Y \\ n_1\dots n_k=\square}} \mu(n_1) \dots \mu(n_k) \prod_{p|2n_1\dots n_k} \frac{p}{p+1} \\
    &\qquad +O\left(X^{\frac{1}{2}+\eps}Y^{k+\eps}\right)+O\left(X^\frac{1}{2}Y^{\frac{5}{4}k}(\log Y)^k\right) \\
    &=\frac{4}{\pi^2}X \sum_{n_1\leq Y, \dots, n_k\leq Y } f(n_1, \dots,n_k) +O\left(X^{\frac{1}{2}+\eps}Y^{k+\eps}\right)+O\left(X^\frac{1}{2}Y^{\frac{5}{4}k}(\log Y)^k\right).
\end{split}
\end{align}

\section{Proof of Theorem \ref{thm:1}}
We separate the argument into two cases, that is, $k=2$ and $k\geq 3$ because we can obtain an elaborated error term and the exact coefficient of the main term in the case $k=2$. In this case we prove it by Perron's formula.
\begin{proof}[Proof of Theorem \ref{thm:1} in the case $k=2$]
We start from (\ref{S_k with f}) with $k=2$ to find that 
\begin{align*}
    S_2(X,Y) &=\frac{4}{\pi^2}X \sum_{n_1\leq Y, n_2\leq Y } f(n_1,n_2) +O\left(X^{\frac{1}{2}+\eps}Y^{2+\eps}\right)+O\left(X^\frac{1}{2}Y^{\frac{5}{2}}(\log Y)^2\right).
\end{align*}
Since $f(n_1, n_2)$ is supported on square-free integers with $n_1n_2=\square$, the non-diagonal terms vanish. Hence we see that 
\begin{align*}
    \sum_{n_1\leq Y, n_2\leq Y } f(n_1,n_2) &= \sum_{n\leq Y} \mu(n)^2 \prod_{p|n^2}\frac{p}{p+1}.
\end{align*}
Set $g(n):= \mu(n)^2 \prod_{p|n^2}\frac{p}{p+1}$ and consider its generating Dirichlet series
\[
G(s)=\sum_{n=1}^\infty \frac{g(n)}{n^s}.
\]
Then by the Euler product formula, we have
\begin{align*}
    G(s) &= \prod_{p} \left(1+\frac{p}{p+1}\frac{1}{p^s}\right) \\
    &= \zeta(s)\prod_{p} \left(1-\frac{1}{p+1}\frac{1}{p^s}-\frac{p}{p+1}\frac{1}{p^{2s}}\right) \\
    &= \zeta(s)E(s),
\end{align*}
say. We see that the Dirichlet series $E(s)$ converges absolutely for $\si>1/2$.

By Perron's formula~\cite[Theorem 5.2 and Corollary 5.3]{MV}, we have
\begin{align*}
    \sideset{}{'}{\sum}_{n \leq Y} g(n) = \frac{1}{2\pi i} \int_{c-iT}^{c+iT} G(s)\frac{Y^s}{s}ds+R,
\end{align*}
with 
\begin{align*}
    R \ll \sum_{\substack{Y/2\leq n\leq 2Y \\ n \neq Y}} \abs{g(n)} \min \left\{1, \frac{Y}{T\abs{Y-n}}\right\}+\frac{(4Y)^c}{T}\sum_{n=1}^\infty \frac{\abs{g(n)}}{n^c},
\end{align*}
where $c=1+1/\log Y$ and $T$ is a parameter which will be chosen later. Here, $\sum^\prime$ indicates that if $Y$ is an integer, then the last term is to be halved. Since $\abs{g(n)} \leq \prod_{p|n^2}\frac{p}{p+1} \leq 1$, we see that
\begin{align}
\label{error-R}
    R \ll 1+\frac{Y\log Y}{T}.
\end{align}

We move the line of integration to $\delta=1/2+1/\log Y$, to have 
\begin{align*}
    \frac{1}{2\pi i} \int_{c-iT}^{c+iT} G(s)\frac{Y^s}{s}ds &= E(1) \\
    &\quad +\frac{1}{2\pi i}\left(\int_{\delta+iT}^{c+iT}+\int_{\delta-iT}^{\delta+iT}-\int_{\delta-iT}^{c-iT}\right)\zeta(s)E(s)\frac{Y^s}{s}ds.
\end{align*}

The Hardy-Littlewood bound of the Riemann zeta-function and the Phragm\'en–Lindel\"of principle imply that
\[
\zeta(\si+it) \ll (1+\abs{t})^{-\frac{\si}{3}+\frac{1}{3}} \quad (1/2 \leq \si < 1).
\]
Using this bound, we have
\begin{align}
\begin{split}
\label{error-H}
       \int_{\delta\pm iT}^{c\pm iT} \abs{\zeta(s)E(s)}\abs{\frac{Y^s}{s}}\abs{ds} &\ll \frac{1}{T^{\frac{5}{6}+\frac{1}{3\log Y}}} \int_{\delta}^{c} Y^\si d\si \\
    &\ll \frac{Y}{T^{\frac{5}{6}}} \exp \left(-\frac{\log T}{3\log Y}\right),
\end{split}
\end{align}
and 
\begin{align}
\begin{split}
\label{error-V}
    \int_{\delta- iT}^{\delta+ iT} \abs{\zeta(s)E(s)}\abs{\frac{Y^s}{s}}\abs{ds} &\ll Y^\frac{1}{2} \int_{1}^T t^{-\frac{2}{3}-\frac{\delta}{3}} dt \\
    &\ll Y^\frac{1}{2} T^{\frac{1}{6}}\exp \left(-\frac{\log T}{3\log Y}\right).
\end{split}
\end{align}

We choose $T=Y^\frac{1}{2}$. Then, we can evaluate that the contributions from (\ref{error-H}) and (\ref{error-V}) are $Y^\frac{7}{12}$. Moreover from (\ref{error-R}) we can evaluate $R \ll Y^\frac{1}{2}\log Y$. 

Lastly, we see that
\[
E(1)=\prod_p \left( 1-\frac{2}{p(p+1)} \right) =Z_2.
\]
Therefore we complete the proof of Theorem \ref{thm:1}.
 \end{proof}
 
\section{Proof of Theorem \ref{thm:2}}\label{section5}
Our main tool to prove Theorem \ref{thm:2} is the de la Bret\`{e}che Tauberian theorem~\cite{B} for non-negative arithmetical functions of several variables. The results of de la Bret\`{e}che were used by Essouabri, Salinas Zavala and T\'oth~\cite{EST22} to give the asymptotic behavior of the average of a wide class of multivariable arithmetic functions. 

We denote the canonical basis of $\mathbb{C}^k$ by $\{\bm{e}_j\}_{j=1}^k$ and its dual basis by $\{\bm{e}_j^*\}_{j=1}^k$.

\begin{thm}[{\cite[Th\'{e}or\`{e}me 1]{B}}]
\label{thm:breteche1}
    Let $f: \mathbb{N}^k \to \mathbb{R}$ be a non-negative function and $F$ the associated Dirichlet series of $
f$ defined by
\[
F(\bm{s}) = F(s_1,\dots,s_k) = \sum_{n_1,\dots,n_k}^\infty \frac{f(n_1,\dots,n_k)}{n^{s_1}\dots n_k^{s_k}}.
\]
Denote by $\mathcal{LR}_k^+(\mathbb{C})$ the set of non-negative $\mathbb{C}$ linear forms from $\mathbb{C}^k$ to $\mathbb{C}$ on $(\mathbb{R}_{\geq 0})^k$. Moreover, assume that there exists $(c_1, \dots,c_k) \in (\mathbb{R}_{\geq 0})^k$ such that:
\begin{enumerate}[(1)]
\item For $\bm{s} \in \mathbb{C}^k$, $F(s_1,\dots,s_k)$ is absolutely convergent for $\re(s_i) > c_i$ for all $1 \leq i \leq k$.

\item There exist a finite family $\mathcal{L} = (l^{(i)})_{1\leq i \leq q}$ of non-zero elements of $\mathcal{LR}_k^+(\mathbb{C})$, a finite family $(h^{
(i)})_{1\leq i \leq q^\prime}$ of elements of $\mathcal{LR}_k^+(\mathbb{C})$ and $\delta_1, \delta_3>0$ such that the function $H$ defined by
\[
H(\bm{s}) = F(\bm{s} + \bm{c}) \prod_{i=1}^q l^{(i)}(\bm{s})
\]
has a holomorphic continuation to the domain
\begin{align*}
D(\delta_1, \delta_3) 
&= \left\{ \bm{s} \in \mathbb{C}^k \mid \re\left(l^{(i)}(\bm{s})\right) > -\delta_1 \text{ for all } \leq i \leq q \right. \\
& \qquad\qquad\qquad \left. \text{ and } \re\left(h^{(i)}(\bm{s})\right)> -\delta_3 \text{ for all } 1\leq i \leq q^\prime \right\}.
\end{align*}
\item There exist $\delta_2>0$ such that for $\eps,\eps^\prime>0$ we have uniformly in $\bm{s} \in D(\delta_1 -\eps,\delta_3-\eps^\prime)$ 
\[
H(\bm{s}) \ll \prod_{i=1}^q \left(\abs{\im\left(l^{(i)}(\bm{s})\right)}+1 \right)^{1-\delta_2 \min \left\{0, \re\left( l^{(i)}(\bm{s})\right)\right\}}(1+(\abs{\im(s_1)}+\dots+\abs{\im(s_k)})^\eps).
\]
\end{enumerate}
Set $J = J(\mathbb{C}) = \{j \in \{1, \dots, k\} \mid c_j = 0\}$. Denote $w$ to be the cardinality of $J$ and by $j_1 <\dots<j_w$ its elements in increasing order. Define the $w$ linear forms $l^{(q+i)}$ $(1 \leq i \leq w)$ by $l^{(q+i)}(\bm{s}) = {e^{*}}_{j_i}(\bm{s}) = s_{j_i}$.

Then, for any $\bm{\beta} = (\beta_1,\dots, \beta_k) \in (0,\infty)^k$, there exist a polynomial $Q_{\bm{\beta}} \in \mathbb{R}[X]$ of degree at most $q + w - Rank \left(l^{(1)}, \dots, l^{(q+w)}\right)$ and $\theta > 0$ such that as $x \to \infty$
\[
\sum_{n_1\leq x^{\beta_1}}\dots\sum_{n_k \leq x^{\beta_k}} f(n_1,\dots,n_k) = x^{\langle \bm{c}, \bm{\beta} \rangle} Q_{\bm{\beta}}(\log x) +O\left( x^{\langle \bm{c}, \bm{\beta} \rangle - \theta}\right).
\]
Here, $\langle \cdot, \cdot \rangle$ denotes the usual dot product in $\mathbb{R}^k$.
\end{thm}

The next theorem gives a determination of the precise degree and the leading coefficient of the polynomial $Q_{\bm{\beta}}$ appearing in the previous theorem. We denote by $\mathbb{R}_*^+$ the set of strictly positive real numbers, the notation $con^*(\{l^{(1)},\dots,l^{(q)}\})$ means $\mathbb{R}_*^+ l^{(1)}+\dots+\mathbb{R}_*^+ l^{(q)}$.

\begin{thm}[{\cite[(iii) and (iv) of Th\'{e}or\`{e}me 2]{B}}]
\label{thm:breteche2}
Let $f : \mathbb{N}^k \to \mathbb{R}$ be a non-negative function satisfying the assumptions of Theorem \ref{thm:breteche1}. Let $\bm{\beta} = (\beta_1,\dots,\beta_k) \in (0,\infty)^k$. Set $\mathcal{B} = \sum_{i=1}^k\beta_i \bm{e}_i^* \in \mathcal{LR}_k^+(\mathbb{C})$.
\begin{enumerate}[(1)]
\item If the Dirichlet series $F$ satisfies the additional two assumptions:
\begin{enumerate}[(C1)]
\item\label{C1}  There exists a function $G$ such that $H(\bm{s}) = G\left(l^{(1)}(\bm{s}),\dots,l^{(q+w)}(\bm{s}) \right)$.
\item\label{C2} $\mathcal{B} \in Vect \left(\{l^{(i)} \mid i = 1,\dots q + w\}\right)$ and there is no subfamily $\mathcal{L}^\prime$ of $\mathcal{L}_0 :=\left( l^{(i)}\right)_{1\leq i \leq q+w}$ such that $\mathcal{L}^\prime \neq \mathcal{L}_0$, $\mathcal{B} \in Vect(\mathcal{L}^\prime)$ and $\#\mathcal{L}^\prime -Rank(\mathcal{L}^\prime) = \#\mathcal{L}_0 - Rank(\mathcal{L}_0)$.
\end{enumerate}
Then, the polynomial $Q_{\bm{\beta}}$ satisfies the relation
\[
Q_{\bm{\beta}}(\log x) = H(\bm{0})x^{-\langle \bm{c}, \bm{\beta} \rangle}\mathcal{I}_{\bm{\beta}}(x)+O((\log x)^{\rho-1}),
\]
where $\rho := q+w-Rank(l^{(1)},\dots, l^{(q+w)})$ and 
\[
\mathcal{I}_{\bm{\beta}}(x) := \int_{\mathcal{A}_{\bm{\beta}}(x)} \frac{dy_1\dots dy_q}{\prod_{i=1}^q y_i^{1-l^{(i)}(\bm{c})}},
\]
with
\[
\mathcal{A}_{\bm{\beta}(x)} := \left\{ \bm{y} \in [1,\infty)^q \middle| \prod_{i=1}^q y_i^{l^{(i)}(\bm{e}_j)} \leq x^{\beta_j} \text{ for all } 1 \leq j \leq k \right\}.
\]
\item If $Rank\left(l^{(1)},\dots,l^{(q+w)} \right)=n, H(\bm{0})\neq 0$ and $\mathcal{B} \in con^* \left(\{ l^{(1)},\dots,l^{(q+w)} \} \right)$, then $\operatorname{deg}(Q_{\bm{\beta}}) =  q + w - n$.
\end{enumerate}
\end{thm}

\begin{proof}[Proof of Theorem \ref{thm:2}]
From (\ref{euler-prod}), we can easily find that $f(n_1,\dots,n_k)$ is non-negative and in Lemma \ref{lem:convolution}, we proved that $f(n_1, \dots,n_k)$ has an absolutely convergent series $F(s)$ for $\re(s_j) > 1/2$ for all $1 \leq j \leq k$. This shows that $f(n_1, \dots,n_k)$ satisfies (1) of Theorem \ref{thm:breteche1}. 

Next, we write $\bm{1/2} = (1/2,\dots,1/2)$. Then $F(\bm{s} + \bm{1/2})$ is an absolutely convergent series for $\re(s_j)> 0$. Therefore, we define the function
\begin{align}
\label{H(s)}
H(\bm{s}):=F(\bm{s} + \bm{1/2})\prod_{\substack{I \subseteq \{1,\dots,k\} \\ \abs{I}=2}}s_I.    
\end{align}
So, we take $\mathcal{L} = \{ s_i+s_j \mid i,j \in \{1,\dots,k\} \}$ in (2) of Theorem \ref{thm:breteche1}.
Then, by Lemma \ref{lem:convolution} it can be rewritten as
\begin{align*}
    H(\bm{s})&=\left( \prod_{\substack{I \subseteq \{1,\dots,k\} \\ \abs{I} =2}} \zeta\left(s_I+\frac{\abs{I}}{2}\right)s_I\right)E(\bm{s}+\bm{1/2}).
\end{align*}
For $I \subseteq \{1,\dots,k\}$ such that $\abs{I}=2$, then there exists $\delta_1 \in (0,1/4)$ such that $\zeta\left(s_I+\abs{s_I}/2\right)s_I = \zeta\left(s_I+1\right)s_I$ has analytic continuation to the plane $\re(s_j)>-\delta_1$ for $j \in I$ and $\delta_1>0$. Furthermore, $E(\bm{s}+\bm{1/2})$ is holomorphic in $\re(s_j)>-\delta_1$ for $1 \leq j \leq k$. Since $c_j = 1/2$ for all $1 \leq j \leq k$, we can take $h^{(j)}(\bm{s}) = s_j$ in the notation of Theorem \ref{thm:breteche1}. Then we put $\delta_3=\delta_1$. Therefore $f(n_1,\dots,n_k)$ also satisfies (2) of Theorem \ref{thm:breteche1}.

Finally, for $\re(s_j) > -1/4$, by applying the convexity bound of the Riemann zeta-function,
\[
\zeta\left(s_I+\frac{\abs{I}}{2}\right)s_I \ll (\abs{s_I}+1)^{1- \frac{1}{2}\min \{0, \re(s_I)\}}
\]
holds for $\abs{I}=2$. The above argument shows that $H(s)$ satisfies (3) of Theorem \ref{thm:breteche1} with $\delta_2 = 1/2$.

Since $q= \binom{k}{2}, w=0$ and the rank of the collection of linear forms $s_I$ is $k$, we have 
\[
\sum_{n_1,\dots,n_k \leq Y} f(n_1,\dots,n_k) = Y^\frac{k}{2} Q(\log Y) + O(Y^{\frac{k}{2}-\theta}),
\]
where $Q(\log Y)$ is a polynomial of degree at most $\binom{k}{2}-k$.

The remained task is to determine the degree and the leading coefficient of the polynomial $Q$ by using Theorem \ref{thm:breteche2}. Since $\bm{c}=\bm{1/2}$, we know that $w=0$, and then $Rank\left(l^{(1)},\dots,l^{(q)} \right) =k$. Moreover as $s_I \to 0$, $\zeta(s_I+\abs{I}/2)s_I \to 1$ if $\abs{I}=2$. From the Euler product,it holds that $E(\bm{1/2})$ does not vanish. Hence $H(\bm{0})\neq 0$. At last, we see that $\mathcal{B} = \sum_{i=1}^k \bm{e}_i^*(\bm{s}) \in con^* \left(\{ l^{(1)},\dots,l^{(q)} \} \right)$ for $\bm{e}_i^*(\bm{s}) = s_i$. Therefore by Theorem \ref{thm:breteche2} (2), we have $\operatorname{deg}(Q) = \binom{k}{2}-k$. 

We denote the leading coefficient of $Q$ by $Z_k$. We can easily check that the assumptions (C1) and (C2) in Theorem  \ref{thm:breteche2} (1) are satisfied. In fact, by (\ref{E(s)}) and (\ref{H(s)}), it is clear that such a function $G$ exists. Also, for $\mathcal{L}_0 =\left( l^{(i)}\right)_{1\leq i \leq q}$, we already showed that $Rank(\mathcal{L}_0)=k$ and $\#\mathcal{L}_0= \binom{k}{2}$. Hence, there is no subfamily $\mathcal{L}^\prime$ of $\mathcal{L}_0$ such that $\mathcal{L}^\prime \neq \mathcal{L}_0$, $\mathcal{B} \in Vect(\mathcal{L}^\prime)$ and $\#\mathcal{L}^\prime=\#\mathcal{L}_0$.

Hence, by applying Theorem \ref{thm:breteche2} (1), we have 
\begin{align*}
    \mathcal{I} &= \lim_{x \to \infty} \frac{1}{x^{\frac{k}{2}}(\log x)^{\binom{k}{2}-k}}\mathcal{I}(x) = \lim_{x \to \infty} \frac{1}{x^{\frac{k}{2}}(\log x)^{\binom{k}{2}-k}}\int_{\mathcal{A}(x)} dy_1\dots dy_q,
\end{align*}
where
\[
\mathcal{A}(x) = \left\{ (y_1,\dots,y_q) \in [1,\infty)^q \middle| \prod_{i=1}^q y_{i}^{l^{(i)}(\bm{e}_j)} \leq x \text{ for all } 1 \leq j \leq k \right\},
\]
because it holds that $l^{(i)}(\bm{c})=1$ for all $1 \leq i \leq q$ and $l^{(i)}(\bm{e}_j)=1$ for $k-1$ linear forms in $\mathcal{L}$, and $l^{(i)}(\bm{e}_j)=0$ otherwise. 

Moreover, since $\zeta(s_I+\abs{I}/2)s_I \to 1$ as $s_I \to 0$ for $\abs{I}=2$, by Lemma \ref{lem:convolution} we obtain
\begin{align*}
H(\bm{0}) &=\prod_{p} \left( 1-\binom{k}{2} \frac{1}{p(p+1)} + \sum_{h=1}^{\lfloor \frac{k}{2} \rfloor} \sum_{\ell = 1}^{\binom{k}{2}} (-1)^{\ell}\binom{k}{2h} \frac{1}{p^{h+\ell-1}(p+1)} \right).
\end{align*}
Therefore, we complete the proof.
\end{proof}

\begin{rem}
    If $k=2$, de la Bret\`{e}che's method also works. In fact, in the case $k=2$, we take $q=1$ and $l^{(1)}(\bm{s})=s_1+s_2$. Then $Rank\left(l^{(1)} \right) =1$. Thus the degree of the polynomial $Q$ is $\binom{2}{2}-1 =0$. 
\end{rem}

\begin{ack} 
The author would like to thank Professor Kohji Matsumoto for his valuable comments. The author would also like to thank Kittipong Subwattanachai for reading the draft carefully and pointing out many typos. The author is supported by Grant-in-Aid for JSPS Research Fellow (Grant Number:24KJ1235).
\end{ack}

\end{document}